\documentclass[12pt,a4paper]{article}
\usepackage{geometry}
\usepackage{amsmath,amssymb,amsfonts}
\usepackage{amsthm}
\usepackage{mathtools}
\usepackage{marvosym}
\usepackage{authblk}
\usepackage{enumerate}
\usepackage[english]{babel}
\usepackage{enumitem}
\usepackage{todonotes}
\usepackage{indentfirst}
\usepackage{graphicx}
\usepackage{tabularx}
\usepackage{tocbibind}
\usepackage{tikz}
\usepackage{array}
\usepackage{caption}
\usepackage{setspace}
\usepackage{hyperref}
\usepackage{lipsum}
\usepackage{listings}
\usepackage{multicol,multirow}
\usepackage{xcolor,xspace}
\usepackage[english]{babel}
\usepackage[export]{adjustbox}
\newcommand{\keywords}[1]{\textbf{Keywords:} #1}

\newtheorem{theorem}{Theorem}[section]
\newtheorem{corollary}{Corollary}[theorem]
\newtheorem{lemma}[theorem]{Lemma}

\theoremstyle{definition}
\newtheorem{definition}{Definition}[section]

\theoremstyle{remark}

\title{\textbf{On the spectrum of closed neighbourhood corona product of graph and its application}}
\author[1]{Bishal Sonar\thanks{Email: bsonarnits@gmail.com; Orcid link: 0009-0007-6145-6481}}
\author[2]{Ravi Srivastava\thanks{Corresponding author Email: ravi@nitsikkim.ac.in; Orcid link: 0009-0001-9615-6709}}
\affil[1,2]{Department of Mathematics, National Institute of Technology Sikkim, South Sikkim 737139, India}
\date{}
\begin{document}
\parskip1ex
\parindent0pt
\maketitle

\begin{abstract}
    \noindent In this paper, we investigate the spectral properties of the closed neighborhood corona product of graphs, which was introduced by Harishchandra S. Ramane et al.~\cite{ramane2021polynomials} (cf. Polynomials Associated with Closed Neighborhood Corona and Neighborhood Complement Corona of Graphs). Based on their results, such as characteristic polynomials of the adjacency, Laplacian, and signless Laplacian matrices, we further investigate the spectral characteristics of this product graph. Specifically, we investigate conditions under which cospectrality occurs for this operation. Further, we determine the Kirchhoff index and count spanning trees and identify sequences of non-cospectral equienergetic product graphs. Finally, we develop criteria for when the product graph is integral and thereby contribute to a deeper understanding of the algebraic and combinatorial structure of the product graph.
\end{abstract}


\newpage

\section{Introduction}

    In graph theory, a graph $\mathcal{G}$ is denoted as an ordered pair $(\mathcal{V}, \mathcal{E})$, where $\mathcal{V} = \{v_i : i = 1, 2, \ldots, n\}$ represents the set of vertices, and $\mathcal{E} = \{e_i : i = 1, 2, \ldots, m\}$ represents the set of edges. The vertex degree $d(v_i)$ is the number of edges connected to the vertex $v_i$. Hence, the adjacency matrix $\mathcal{A}_{\mathcal{G}} = (a_{ij})_{n\times n}$ is a critical representation of the graph $\mathcal{G}$, where each entry $a_{ij} = 1$ if there is an edge between vertex $v_i$ and vertex $v_j$ i.e., $v_i$ and $v_j$ are adjacent, and $a_{ij} = 0$ otherwise. Additionally, two crucial matrices associated with a graph $\mathcal{G}$: the Laplacian matrix $\mathcal{L}_{\mathcal{G}} = \mathcal{D}_{\mathcal{G}} - \mathcal{A}_{\mathcal{G}}$ and the signless Laplacian matrix $\mathcal{Q}_{\mathcal{G}} = \mathcal{D}_{\mathcal{G}} + \mathcal{A}_{\mathcal{G}}$. The diagonal matrix $\mathcal{D}_{\mathcal{G}}$ has the $i^{th}$ diagonal entry as $d(v_i)$, representing the degree of vertex $v_i$. The shortest distance between any two vertices $v_i$ and $v_j$ of the graph $\mathcal{G}$ is denoted as $d(v_i,v_j)$, and  for any $v_i\in\mathcal{V}$, $N(v_i)=\{v_j\in\mathcal{V}: v_i \text{ and } v_j \text{ are adjacent}\}$. The energy of a graph $\mathcal{G}$ of order $n$ is the sum of the absolute eigenvalues of the adjacency matrix, given by $$\mathcal{E}_{\mathcal{G}}=\sum_{i=1}^n|\lambda_i|,$$ where $\lambda_i's$ are the eigenvalues of $\mathcal{A}_{\mathcal{G}}$. A graph $\mathcal{G}$ is classified as Integral if all the eigenvalues of its adjacency matrix are integers. This discussion focuses exclusively on simple, undirected, and finite graphs.

    Frucht and F Harary\cite{frucht1970corona} first defined the corona product for graphs. Later, Cam McLeman and Erin McNicholas~\cite{mcleman2011spectra} introduced the coronal of graphs to obtain the adjacency spectra for arbitrary graphs. Shu and Gui~\cite{cui2012spectrum} later generalized this concept and defined the coronal of the Laplacian and the signless Laplacian matrix of graphs. Since then, various researchers have explored the variants of corona product for both signed and unsigned graphs~\cite{brondani2023aalpha, gopalapillai2011spectrum, guragain2024spectrum, joseph2023graph, singh2023structural, sonar2023spectrum, sonar2024signed}. 
    We  establish a condition for the closed neighborhood corona product graph to be cospectral. Additionally, we apply our findings to obtain the number of spanning trees, the Kirchhoff Index, the sequence of non-cospectral equienergetic product graphs and the conditions required for the product graph to be Integral.\\
    In the subsequent Section \ref{S2}, we present an array of preliminary results. This is followed by an in-depth exploration of the definition and principal findings in Section \ref{S3}. Further, Section \ref{S4} is dedicated to examining the multifaceted applications of the graph product, providing a comprehensive overview of its utility in academic research.

\section{Preliminaries}\label{S2}    
    \begin{lemma}\label{L2}~\cite{bapat2010graphs}
        Let $\mathcal{A}_1,\mathcal{A}_2,\mathcal{A}_3,$ and $\mathcal{A}_4$ be matrix of order $a_1\times a_1,~ a_1\times a_2, ~a_2\times a_1, ~a_2\times a_2$ respectively. If $\mathcal{A}_1$ and $\mathcal{A}_4$ are invertible, then 
            \begin{align*}
                \det\begin{bmatrix}
                \mathcal{A}_1&\mathcal{A}_2\\\mathcal{A}_3&\mathcal{A}_4
        \end{bmatrix}&=\det(\mathcal{A}_1)\det(\mathcal{A}_4-\mathcal{A}_3\mathcal{A}_1^{-1}\mathcal{A}_2)\\
        &=\det(\mathcal{A}_4)\det(\mathcal{A}_1-\mathcal{A}_2\mathcal{A}_4^{-1}\mathcal{A}_3).
        \end{align*}
    \end{lemma}

        Consider two matrices $\mathcal{A}=(a_{ij})$ with dimensions $m_1 \times n_1$ and $\mathcal{B}=(b_{ij})$ with dimensions $m_2 \times n_2$. The operation known as the Kronecker product, symbolized by $\mathcal{A}\otimes \mathcal{B}$, yields a new matrix with dimensions $m_1m_2 \times n_1n_2$. This procedure involves taking each element $a_{ij}$ from matrix $\mathcal{A}$ and multiplying it by the entire matrix $\mathcal{B}$, effectively resizing $\mathcal{B}$ by the factor of $a_{ij}$. As highlighted by Neumaier (1992)~\cite{neumaier1992horn}, this technique allows for the assembling two smaller matrices.

        \textbf{Properties:}
            \begin{itemize}
                \item The operation is associative.
                \item Transposing the Kronecker product results in $(\mathcal{A}\otimes \mathcal{B})^T=\mathcal{A}^T\otimes \mathcal{B}^T$.
                \item The product of two Kronecker products can be expressed as $(\mathcal{A}\otimes \mathcal{B})(\mathcal{C} \otimes \mathcal{D})=\mathcal{AC}\otimes \mathcal{BD}$, assuming the products $\mathcal{AC}$ and $\mathcal{BD}$ are defined.
                \item The inverse of a Kronecker product is given by $(\mathcal{A}\otimes \mathcal{B})^{-1}=\mathcal{A}^{-1}\otimes \mathcal{B}^{-1}$, provided that $\mathcal{A}$ and $\mathcal{B}$ are invertible matrices.
                \item If $\mathcal{A}$ and $\mathcal{B}$ are square matrices of sizes $m \times m$ and $n \times n$ respectively, then the determinant of their Kronecker product is $\det(\mathcal{A} \otimes \mathcal{B})=(\det \mathcal{A})^n(\det \mathcal{B})^m.$
            \end{itemize}

        \subsection{Definition\texorpdfstring{~\cite{mcleman2011spectra}}{}}\label{Def2}
            Let $\mathcal{G}$ be a graph with $n$ vertices. The coronal of $\mathcal{G}$, denoted by $\chi_{\mathcal{G}}(x)$, is defined as the sum of all the entries in the matrix $(xI_n - \mathcal{A}_{\mathcal{G}})^{-1}$, where $\mathcal{A}_{\mathcal{G}}$ is the adjacency matrix of the graph $\mathcal{G}$, and $I_n$ is the identity matrix of order $n$.
            \begin{equation}
                \chi_{\mathcal{G}}(x)=\textbf{1}_n^T(xI_n-\mathcal{A}_{\mathcal{G}})^{-1}\textbf{1}_n,
            \end{equation}
            where $\textbf{1}_n$ is all one column matrix.\\
            The coronal of Laplacian and signless Laplacian matrix was introduced by Gui \cite{cui2012spectrum} and is given by $\chi_{\mathcal{L}_{\mathcal{G}}}=1_n(xI_n-\mathcal{L}_{\mathcal{G}})^{-1}1_n$ and $\chi_{\mathcal{Q}_{\mathcal{G}}}=1_n(xI_n-\mathcal{Q}_{\mathcal{G}})^{-1}1_n$ respectively.
        \subsection{Notation}
            We use the following notation for a graph $\mathcal{G}$ of order n.\\
            Adjacency Spectrum $\sigma(\mathcal{A})=\{\lambda_1,\lambda_2,\ldots,\lambda_n\}$\\
            Laplacian Spectrum  $\sigma(\mathcal{L})=\{\gamma_1,\gamma_2,\ldots,\gamma_n\}$\\
            Signless Laplacian Spectrum $\sigma(\mathcal{Q})=\{\nu_1,\nu_2,\ldots,\nu_n\}.$
            
\section{Spectral properties}\label{S3}
        
    \begin{definition}(Closed neighbourhood corona)~\label{Def3}
        Consider two graphs $\mathcal{G}_1$ and $\mathcal{G}_2$ with $n_1$ and $n_2$ vertices and $e_1$ and $e_2$ edges respectively. The closed neighbourhood corona of $\mathcal{G}_1$ and $\mathcal{G}_2$, denoted by $\mathcal{G}_1\boxtimes\mathcal{G}_2$, is a new graph obtained by creating $n_1$ copies of $\mathcal{G}_2$. Each vertex of the $i^{th}$ copy of $\mathcal{G}_2$ is then connected to the $i^{th}$ vertex and neighbourhood of the $i^{th}$ vertex ($u_i$) of $\mathcal{G}_1$.
    \end{definition}
    The number of vertices in $\mathcal{G}_1\boxtimes\mathcal{G}_2$ is $n_1(1+n_2)$ where as the number of edges in $\mathcal{G}_1\boxtimes\mathcal{G}_2$ is $e_1+n_1e_2+2e_1n_2+n_1n_2$.
    
    \begin{figure}
        \centering
        \begin{tikzpicture}
\node (u1) at (-11, 0) [circle, fill=black, inner sep=1.5pt, label=above:$u_1$] {};
\node (u2) at (-10, 0) [circle, fill=black, inner sep=1.5pt, label=above:$u_2$] {};
\node (u3) at (-10, -1) [circle, fill=black, inner sep=1.5pt, label=below:$u_3$] {};
\node (u4) at (-11, -1) [circle, fill=black, inner sep=1.5pt, label=below:$u_4$] {};
\node at (-10.5, -1.7) {$G_1$};

\draw (u1) -- (u2);
\draw (u3) -- (u2);
\draw (u4) -- (u3);
\draw (u1) -- (u4);

\node (v1) at (-10.5, -3) [circle, fill=black, inner sep=1.5pt, label=above:$v_1$] {};
\node (v2) at (-11, -4) [circle, fill=black, inner sep=1.5pt, label=below:$v_2$] {};
\node (v3) at (-10, -4) [circle, fill=black, inner sep=1.5pt, label=below:$v_3$] {};
\node at (-10.5, -4.7) {$G_2$};

\draw (v1) -- (v2);
\draw (v3) -- (v2);
\draw (v1) -- (v3);


\node (v21) at (-8, 1) [circle, fill=black, inner sep=1.5pt, label=above:$v_2^1$] {};
\node (v32) at (-3, 1) [circle, fill=black, inner sep=1.5pt, label=above:$v_3^2$] {};
\node (v31) at (-9, 0.5) [circle, fill=black, inner sep=1.5pt, label=above:$v_3^1$] {};
\node (v22) at (-2, 0.5) [circle, fill=black, inner sep=1.5pt, label=above:$v_2^2$] {};
\node (v11) at (-8, 0) [circle, fill=black, inner sep=1.5pt, label=below:$v_1^1$] {};
\node (v12) at (-3, 0) [circle, fill=black, inner sep=1.5pt, label=below:$v_1^2$] {};
\node (u1) at (-6.5, -.5) [circle, fill=black, inner sep=1.5pt, label={330:$u_1$}] {};
\node (u2) at (-4.5, -.5) [circle, fill=black, inner sep=1.5pt, label={240:$u_2$}] {};
\node (u3) at (-4.5, -2.5) [circle, fill=black, inner sep=1.5pt, label={120:$u_3$}] {};
\node (u4) at (-6.5, -2.5) [circle, fill=black, inner sep=1.5pt, label={60:$u_4$}] {};
\node (v34) at (-8, -4) [circle, fill=black, inner sep=1.5pt, label=below:$v_3^4$] {};
\node (v23) at (-3, -4) [circle, fill=black, inner sep=1.5pt, label=below:$v_2^3$] {};
\node (v24) at (-9, -3.5) [circle, fill=black, inner sep=1.5pt, label=below:$v_2^4$] {};
\node (v33) at (-2, -3.5) [circle, fill=black, inner sep=1.5pt, label=below:$v_3^3$] {};
\node (v14) at (-8, -3) [circle, fill=black, inner sep=1.5pt, label=above:$v_1^4$] {};
\node (v13) at (-3, -3) [circle, fill=black, inner sep=1.5pt, label=above:$v_1^3$] {};
\node at (-5.5, -4.7) {$G_1\boxtimes G_2$};

\draw (u1) -- (u2);
\draw (u3) -- (u2);
\draw (u4) -- (u3);
\draw (u1) -- (u4);
\draw (v31) -- (v21);
\draw (v11) -- (v21);
\draw (v11) -- (v31);
\draw (v32) -- (v22);
\draw (v12) -- (v22);
\draw (v12) -- (v32);
\draw (v33) -- (v23);
\draw (v13) -- (v23);
\draw (v13) -- (v33);
\draw (v34) -- (v24);
\draw (v14) -- (v24);
\draw (v14) -- (v34);
\draw (v31) to[out=270,in=220] (u1);
\draw (v21) -- (u1);
\draw (v11) -- (u1);
\draw (v32) -- (u2);
\draw (v22) to[out=260,in=320] (u2);
\draw (v12) -- (u2);
\draw (v33) to[out=90,in=60] (u3);
\draw (v23) -- (u3);
\draw (v13) -- (u3);
\draw (v34) -- (u4);
\draw (v24) to[out=90,in=120] (u4);
\draw (v14) -- (u4);
\draw (v31) -- (u2);
\draw (v21) -- (u2);
\draw (v11) -- (u2);
\draw (v31) -- (u4);
\draw (v21) -- (u4);
\draw (v11) -- (u4);
\draw (v32) -- (u1);
\draw (v22) -- (u1);
\draw (v12) -- (u1);
\draw (v32) -- (u3);
\draw (v22) -- (u3);
\draw (v12) -- (u3);
\draw (v33) -- (u2);
\draw (v23) -- (u2);
\draw (v13) -- (u2);
\draw (v33) -- (u4);
\draw (v23) -- (u4);
\draw (v13) -- (u4);
\draw (v34) -- (u1);
\draw (v24) -- (u1);
\draw (v14) -- (u1);
\draw (v34) -- (u3);
\draw (v24) -- (u3);
\draw (v14) -- (u3);

\end{tikzpicture}
        \caption{Closed neighbourhood corona product of $C_4$ and $C_3$.}
        \label{fig:1}
    \end{figure}
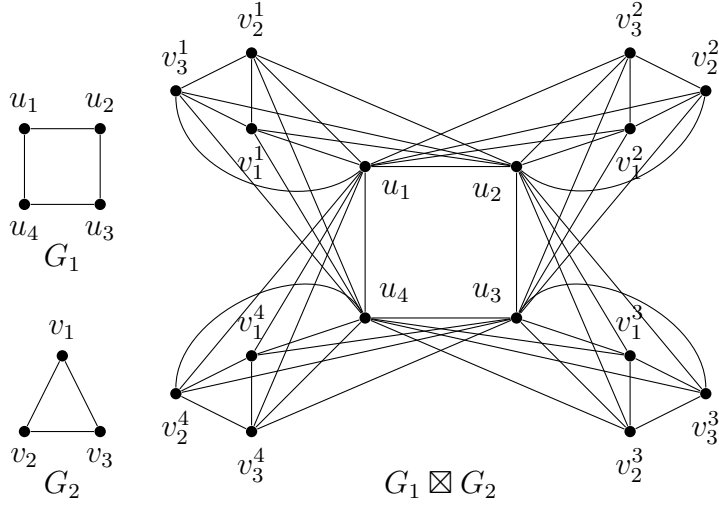
    Let $\mathcal{G}_1$ and $\mathcal{G}_1$ be arbitrary graphs on $n_1$ and $n_2$ vertices, respectively. Following the labelling given by I. Gopalapillai \cite{gopalapillai2011spectrum}, we will label the vertices of $\mathcal{G}_1\boxtimes\mathcal{G}_2$ as follows. Let $V(\mathcal{G}_1)=\{u_i:i=1,2,\ldots,n_1\},$ Let $\mathcal{G}_1$ and $\mathcal{G}_2$ be arbitrary graphs with $n_1$ and $n_2$ vertices, respectively. Following the vertex labelling approach introduced by I. Gopalapillai \cite{gopalapillai2011spectrum}, we define the labelling of the vertices of the closed neighbourhood corona product $\mathcal{G}_1\boxtimes\mathcal{G}_2$. Denote the vertex sets of $\mathcal{G}_1$ and $\mathcal{G}_2$ as $V(\mathcal{G}_1)=\{u_i : i=1,2,\ldots,n_1\}$ and $V(\mathcal{G}_2)=\{v_i : i=1,2,\ldots,n_2\}$, respectively. For each vertex $u_i$ in $\mathcal{G}_1$, we associate a copy of $\mathcal{G}_2$, denoted as the set $\{v_1^i, v_2^i, \ldots, v_{n_2}^i\}$, where $v_j^i$ represents the vertex in the $i^{th}$ copy of $\mathcal{G}_2$ corresponding to the vertex $v_j$ in $\mathcal{G}_2$. The labelling scheme facilitates the analysis of the product graph $\mathcal{G}_1\boxtimes\mathcal{G}_2$ by systematically identifying the vertices across the two graph components. Denote
    \begin{equation} \label{eqn2}
                V_i=\big\{v_i^1,v_i^2,\ldots,v_i^{n_1}\big\}, ~~~i=1,2,\ldots,n_2.
    \end{equation}
    Then $V(\mathcal{G}_1)\cup V_1\cup V_2\cup\ldots\cup V_{n_2}$ is a partition of $V(\mathcal{G}_1\boxtimes\mathcal{G}_2)$.\\
    
\subsection{Adjacency spectrum}
    \begin{theorem}\label{Th1}
        Let $\mathcal{G}_1$ and $\mathcal{G}_2$ be two graphs of orders $n_1$ and $n_2$, respectively, each with their associated eigenvalues denoted by $\lambda_{i1},\lambda_{i2},\ldots,\lambda_{in_i}$ for $i=1,2$. The adjacency characteristic polynomial of the graph $\mathcal{G}_1\boxtimes\mathcal{G}_2$ can be formulated as follows: \\ $f(\mathcal{A}_{\mathcal{G}_1\boxtimes\mathcal{G}_2},x)=\prod_{i=1}^{n_2}(x-\lambda_{2i})^{n_1}\cdot\prod_{i=1}^{n_1}\big[x-\lambda_{1i}-\chi_{A_{\mathcal{G}_2}}(x)(1+\lambda_{1i}^2)\big].$  
    \end{theorem}
    \begin{proof}
        From the partition generated by equation (\ref{eqn2}), the adjacency matrix of $\mathcal{G}_1\boxtimes\mathcal{G}_2$ is,\\
        $$\mathcal{A}_{\mathcal{G}_1\boxtimes\mathcal{G}_2}=\begin{bmatrix}
            \mathcal{A}_{\mathcal{G}_1}&\textbf{1}^T_{n_2}\otimes (I_{n_1}+\mathcal{A}_{\mathcal{G}_1}))\\
            \textbf{1}_{n_2}\otimes (I_{n_1}+\mathcal{A}_{\mathcal{G}_1}))&A_{\mathcal{G}_2}\otimes I_{n_1}
        \end{bmatrix}$$
        The characteristics polynomial of the adjacency matrix is,
        \begin{equation}
            \begin{split}
                f(\mathcal{A}_{\mathcal{G}_1\boxtimes\mathcal{G}_2},x)&=\det(xI_{n_1(1+n_2)}-\mathcal{A}_{\mathcal{G}_1\boxtimes\mathcal{G}_2})\\
                &=\det\begin{bmatrix}
                       xI_{n_1}-\mathcal{A}_{\mathcal{G}_1}&-\textbf{1}^T_{n_2}\otimes (I_{n_1}+\mathcal{A}_{\mathcal{G}_1})\\
                        -\textbf{1}_{n_2}\otimes (I_{n_1}+\mathcal{A}_{\mathcal{G}_1}))&(xI_{n_2}-A_{\mathcal{G}_2})\otimes I_{n_1}  
                \end{bmatrix}\\
                &=\det[\{xI_{n_2}-A_{\mathcal{G}_2}\}\otimes I_{n_1}]\cdot \det[B_A]\\
                &=\det[xI_{n_2}-A_{\mathcal{G}_2}]^{n_1}\cdot\det[I_{n_1}]^{n_2}\cdot \det[B_A]\\
                &=\prod_{i=1}^{n_2}(x-\lambda_{2i})^{n_1}\cdot\det[B_A]
            \end{split}
        \end{equation}
        where \begin{equation*}
            \begin{split}
                \det(B_A)&=\det\Big[xI_{n_1}-\mathcal{A}_{\mathcal{G}_1}-\textbf{1}^T_{n_2}(xI_{n_2}-A_{\mathcal{G}_2})^{-1}\textbf{1}_{n_2}\otimes(I_{n_1}+\mathcal{A}_{\mathcal{G}_1})^2\Big]\\
                &=\det\Big[xI_{n_1}-\mathcal{A}_{\mathcal{G}_1}-\chi_{A_{\mathcal{G}_2}}(x)(I_{n_1}+\mathcal{A}_{\mathcal{G}_1})^2\Big]\\
                &=\prod_{i=1}^{n_1}\Big[x-\lambda_{1i}-\chi_{A_{\mathcal{G}_2}}(x)(1+\lambda_{1i})^2 \Big].
            \end{split}
        \end{equation*}
        Therefore, $f(\mathcal{A}_{\mathcal{G}_1\boxtimes\mathcal{G}_2},x)=\prod_{i=1}^{n_2}(x-\lambda_{2i})^{n_1}\cdot\prod_{i=1}^{n_1}\Big[x-\lambda_{1i}-\chi_{A_{\mathcal{G}_2}}(x)(1+\lambda_{1i})^2 \Big].$
    \end{proof}

\textbf{Note:} The above result has already been proved by Harishchandra S. Ramane et al.~\cite{ramane2021polynomials}, we are adding this theorem for the sake of proper understanding of the further application.

    \begin{corollary}\label{C1}
        Let $\mathcal{G}_1$ and $\mathcal{G}_2$ be two cospectral graphs of order $n$ having adjacency eigenvalues \\ $\{\lambda_{i1},\lambda_{i2},\ldots,\lambda_{in}\},$ $i=1,2$ and $\mathcal{G}$ be any arbitrary graph of order $m$ having adjacency eigenvalues\\ $\{\lambda_1,\lambda_2,\ldots,\lambda_m\}$, then 
            \begin{enumerate}
                \item $\mathcal{G}_1\boxtimes \mathcal{G}$ and $\mathcal{G}_2\boxtimes \mathcal{G}$ are adjacency cospectral.
                \item $\mathcal{G}\boxtimes \mathcal{G}_1$ and $\mathcal{G}\boxtimes \mathcal{G}_2$ are adjacency cospectral if $\chi_{\mathcal{A}_{\mathcal{G}_1}}(x)=\chi_{A_{\mathcal{G}_2}}(x).$
            \end{enumerate}
    \end{corollary}
    \begin{proof}
        It is easily seen that,
        \begin{equation*}
            \begin{split}
                f(\mathcal{A}_{\mathcal{G}_1\boxtimes\mathcal{G}},x)&=\prod_{i=1}^{m}(x-\lambda_{i})^{n}\cdot\prod_{i=1}^{n}\Big[x-\lambda_{1i}-\chi_{A_{\mathcal{G}}}(x)(1+\lambda_{1i})^2 \Big]\\
                &=\prod_{i=1}^{m}(x-\lambda_{i})^{n}\cdot\prod_{i=1}^{n}\Big[x-\lambda_{2i}-\chi_{A_{\mathcal{G}}}(x)(1+\lambda_{2i})^2 \Big]\\
                &=f(\mathcal{A}_{\mathcal{G}_2\boxtimes\mathcal{G}},x).
            \end{split}
        \end{equation*}
        Similarly,
        \begin{equation*}
            \begin{split}
                f(\mathcal{A}_{\mathcal{G}\boxtimes\mathcal{G}_1},x)&=\prod_{i=1}^{n}(x-\lambda_{1i})^{m}\cdot\prod_{i=1}^{m}\Big[x-\lambda_{i}-\chi_{A_{\mathcal{G}_1}}(x)(1+\lambda_{i})^2\Big]\\
                &=\prod_{i=1}^{n}(x-\lambda_{2i})^{m}\cdot\prod_{i=1}^{m}\Big[x-\lambda_{i}-\chi_{A_{\mathcal{G}_2}}(x)(1+\lambda_{i})^2 \Big]\\
                &=f(\mathcal{A}_{\mathcal{G}\boxtimes\mathcal{G}_2},x).
            \end{split}
        \end{equation*}
    \end{proof}

\subsection{Laplacian spectrum}
    \begin{theorem}\label{Th2}
        Let $\mathcal{G}_1$ be a $r_1$-regular graph of order $n_1$, and $\mathcal{G}_2$ be an arbitrary graph of order $n_2$ with their associated Laplacian eigenvalues $\gamma_{i1},\gamma_{i2},\ldots,\gamma_{in_i}$, $i=1,2.$ Then the Laplacian characteristics polynomial of $\mathcal{G}_1\boxtimes\mathcal{G}_2$ is\\ $f(\mathcal{L}_{\mathcal{G}_1\boxtimes\mathcal{G}_2},x)=\prod_{i=1}^{n_2}(x-(r_1+1)-\gamma_{2i})^{n_1}\cdot\prod_{i=1}^{n_1}\Big[x-n_2(r_1+1)-\gamma_{1i}-\chi_{\mathcal{L}_{\mathcal{G}_2}}(x-(r_1+1))(r_1+1-\gamma_{1i})^2\Big].$ 
    \end{theorem}
    \begin{proof}
        From the partition generated by equation (\ref{eqn2}), the adjacency matrix of $\mathcal{G}_1\boxtimes\mathcal{G}_2$ is,\\
        $$\mathcal{A}_{\mathcal{G}_1\boxtimes\mathcal{G}_2}=\begin{bmatrix}
            \mathcal{A}_{\mathcal{G}_1}&\textbf{1}^T_{n_2}\otimes (I_{n_1}+\mathcal{A}_{\mathcal{G}_1}))\\
            \textbf{1}_{n_2}\otimes (I_{n_1}+\mathcal{A}_{\mathcal{G}_1}))&\mathcal{A}_{\mathcal{G}_2}\otimes I_{n_1}
        \end{bmatrix}$$
        and the diagonal degree matrix of $\mathcal{G}_1\boxtimes\mathcal{G}_2$ is,\\
        $$\mathcal{D}_{\mathcal{G}_1\boxtimes\mathcal{G}_2}=\begin{bmatrix}
            (n_2+1)\mathcal{D}_{\mathcal{G}_1}+n_2I_{n_1}&\textbf{0}_{n_1\times n_1n_2}\\
            \textbf{0}_{n_1n_2\times n_1}&\mathcal{D}_{\mathcal{G}_2}\otimes I_{n_1}+I_{n_2}\otimes (\mathcal{D}_{\mathcal{G}_1}+I_{n_1})
        \end{bmatrix}$$
        Then the Laplacian matrix is,\
        $$\mathcal{L}_{\mathcal{G}_1\boxtimes\mathcal{G}_2}=\begin{bmatrix}
            n_2(\mathcal{D}_{\mathcal{G}_1}+I_{n_1})+\mathcal{L}_{\mathcal{G}_1}&-\textbf{1}^T_{n_2}\otimes (I_{n_1}+\mathcal{A}_{\mathcal{G}_1})\\
            -\textbf{1}_{n_2}\otimes (I_{n_1}+\mathcal{A}_{\mathcal{G}_1}))&\mathcal{L}_{\mathcal{G}_2}\otimes I_{n_1}+I_{n_2}\otimes(\mathcal{D}_{\mathcal{G}_1}+I_{n_1})
        \end{bmatrix}$$
        The characteristics polynomial of the Laplacian matrix is,
        \begin{equation}
            \begin{split}
                f(\mathcal{L}_{\mathcal{G}_1\boxtimes\mathcal{G}_2},x)&=\det(xI_{n_1(1+n_2)}-\mathcal{L}_{\mathcal{G}_1\boxtimes\mathcal{G}_2})\\
                &=\det\begin{bmatrix}
                         (x-n_2r_1-n_2)I_{n_1}-\mathcal{L}_{\mathcal{G}_1}&-\textbf{1}^T_{n_2}\otimes (I_{n_1}+\mathcal{A}_{\mathcal{G}_1})\\
                         -\textbf{1}_{n_2}\otimes (I_{n_1}+\mathcal{A}_{\mathcal{G}_1}) &(x-r_1-1)I_{n_2}-\mathcal{L}_{\mathcal{G}_2}\otimes I_{n_1}
                \end{bmatrix}\\
                &=\det[\{(x-r_1-1)I_{n_2}-\mathcal{L}_{\mathcal{G}_2}\}\otimes I_{n_1}]\cdot \det[B_L]\\
                &=\det[(x-r_1-1)I_{n_2}-\mathcal{L}_{\mathcal{G}_2}]^{n_1}\cdot\det[I_{n_1}]^{n_2}\cdot \det[B_L]\\
                &=\prod_{i=1}^{n_2}(x-(r_1+1)-\gamma_{2i})^{n_1}\cdot\det[B_L]
            \end{split}
        \end{equation}
        where \begin{equation*}
            \begin{split}
                \det[B_L]&=\det\Big[(x-n_2r_1-n_2)I_{n_1}-\mathcal{L}_{\mathcal{G}_1}-\big(\textbf{1}^T_{n_2}\otimes (I_{n_1}+\mathcal{A}_{\mathcal{G}_1})\big)\{(x-r_1-1)I_{n_2}\\&~~~~-\mathcal{L}_{\mathcal{G}_2}\}^{-1}\otimes I_{n_1}\big(\textbf{1}_{n_2}\otimes (I_{n_1}+\mathcal{A}_{\mathcal{G}_1})\big)\Big]\\
                &=\det\Big[(x-n_2r_1-n_2)I_{n_1}-\mathcal{L}_{\mathcal{G}_1}-\chi_{\mathcal{L}_{\mathcal{G}_2}}(x-r_1-1)(I_{n_1}+\mathcal{A}_{\mathcal{G}_1})^2\Big]\\
                &=\prod_{i=1}^{n_1}\Big[x-n_2(r_1+1)-\gamma_{1i}-\chi_{\mathcal{L}_{\mathcal{G}_2}}(x-r_1-1)(1+\lambda_{1i})^2\Big]\\
                &=\prod_{i=1}^{n_1}\Big[x-n_2(r_1+1)-\gamma_{1i}-\chi_{\mathcal{L}_{\mathcal{G}_2}}(x-(r_1+1))(r_1+1-\gamma_{1i})^2\Big]
            \end{split}
        \end{equation*}
        Therefore, $f(\mathcal{L}_{\mathcal{G}_1\boxtimes\mathcal{G}_2},x)=\prod_{i=1}^{n_2}(x-(r_1+1)-\gamma_{2i})^{n_1}\cdot\prod_{i=1}^{n_1}\Big[x-n_2(r_1+1)-\gamma_{1i}-\chi_{\mathcal{L}_{\mathcal{G}_2}}(x-(r_1+1))(r_1+1-\gamma_{1i})^2\Big].$        
    \end{proof}
\textbf{Note:} The above result has already been proved by Harishchandra S. Ramane et al.~\cite{ramane2021polynomials}, we are adding this theorem for the sake of proper understanding of the further application.

    \begin{corollary}\label{C2}
        Let $\mathcal{G}_1$ and $\mathcal{G}_2$ be two Laplacian cospectral graphs of order $n$ having Laplacian eigenvalues $\{\gamma_{i1},\gamma_{i2},\ldots,\gamma_{in}\},$ $i=1,2$ and $\mathcal{G}$ be any arbitrary graph of order $m$ having Laplacian eigenvalues $\{\gamma_1,\gamma_2,\ldots,\gamma_m\}$, then 
            \begin{enumerate}
                \item $\mathcal{G}_1\boxtimes \mathcal{G}$ and $\mathcal{G}_2\boxtimes \mathcal{G}$ are Laplacian cospectral.
                \item $\mathcal{G}\boxtimes \mathcal{G}_1$ and $\mathcal{G}\boxtimes \mathcal{G}_2$ are Laplacian cospectral if $\chi_{L_{\mathcal{G}_1}}(x)=\chi_{L_{\mathcal{G}_2}}(x).$
            \end{enumerate}
    \end{corollary}
    \begin{proof}
        It is easily seen that,
        \begin{equation*}
            \begin{split}
                f(\mathcal{L}_{\mathcal{G}_1\boxtimes\mathcal{G}},x)&=\prod_{i=1}^{m}(x-(r_1+1)-\gamma_{i})^{n}\cdot\prod_{i=1}^{n}\Big[x-m(r_1+1)-\gamma_{1i}-\chi_{\mathcal{L}_{\mathcal{G}_2}}(x-(r_1+1))\\&~~~~(r_1+1-\gamma_{1i})^2\Big]\\
                &=\prod_{i=1}^{m}(x-(r_1+1)-\gamma_{i})^{n}\cdot\prod_{i=1}^{n}\Big[x-m(r_1+1)-\gamma_{2i}-\chi_{\mathcal{L}_{\mathcal{G}_2}}(x-(r_1+1))\\&~~~~(r_1+1-\gamma_{2i})^2\Big]\\
                &=f(\mathcal{L}_{\mathcal{G}_2\boxtimes\mathcal{G}},x).
            \end{split}
        \end{equation*}
        Similarly,
        \begin{equation*}
            \begin{split}
                f(\mathcal{L}_{\mathcal{G}\boxtimes\mathcal{G}_1},x)&=\prod_{i=1}^{n}(x-(r_1+1)-\gamma_{1i})^{m}\cdot\prod_{i=1}^{m}\Big[x-n(r_1+1)-\gamma_{i}-\chi_{\mathcal{L}_{\mathcal{G}_1}}(x-(r_1+1))\\&~~~~(r_1+1-\gamma_{i})^2\Big]\\
                &=\prod_{i=1}^{n}(x-(r_1+1)-\gamma_{2i})^{m}\cdot\prod_{i=1}^{m}\Big[x-n(r_1+1)-\gamma_{i}-\chi_{\mathcal{L}_{\mathcal{G}_2}}(x-(r_1+1))\\&~~~~(r_1+1-\lambda_{i})^2\Big]\\
                &=f(\mathcal{L}_{\mathcal{G}\boxtimes\mathcal{G}_2},x).
            \end{split}
        \end{equation*}
    \end{proof}

\subsection{Signless Laplacian spectrum}
    \begin{theorem}\label{Th3}
        Let $\mathcal{G}_1$ be a $r_1$-regular graph of order $n_1$, and $\mathcal{G}_2$ be an arbitrary graph of order $n_2$ with their associated signless Laplacian eigenvalues $\nu_{i1},\nu_{i2},\ldots,\nu_{in_i}$, $i=1,2.$ Then the Signless Laplacian characteristics polynomial of $\mathcal{G}_1\boxtimes\mathcal{G}_2$ is\\ $f(\mathcal{Q}_{\mathcal{G}_1\boxtimes\mathcal{G}_2},x)=\prod_{i=1}^{n_2}(x-(r_1+1)-\nu_{2i})^{n_1}\cdot\prod_{i=1}^{n_1}\big[x-n_2(r_1+1)-\nu_{1i}-\chi_{\mathcal{Q}_{\mathcal{G}_2}}(x-(r_1+1))(1-r_1+\nu_{1i})^2\big].$  
    \end{theorem}
    \begin{proof}
        From the partition generated by equation (\ref{eqn2}), the adjacency matrix of $\mathcal{G}_1\boxtimes\mathcal{G}_2$ is,\\
        $$\mathcal{A}_{\mathcal{G}_1\boxtimes\mathcal{G}_2}=\begin{bmatrix}
            \mathcal{A}_{\mathcal{G}_1}&\textbf{1}^T_{n_2}\otimes (I_{n_1}+\mathcal{A}_{\mathcal{G}_1}))\\
            \textbf{1}_{n_2}\otimes (I_{n_1}+\mathcal{A}_{\mathcal{G}_1}))&\mathcal{A}_{\mathcal{G}_2}\otimes I_{n_1}
        \end{bmatrix}$$
        and the diagonal degree matrix of $\mathcal{G}_1\boxtimes\mathcal{G}_2$ is,\\
        $$\mathcal{D}_{\mathcal{G}_1\boxtimes\mathcal{G}_2}=\begin{bmatrix}
            (n_2+1)\mathcal{D}_{\mathcal{G}_1}+n_2I_{n_1}&\textbf{0}_{n_1\times n_1n_2}\\
            \textbf{0}_{n_1n_2\times n_1}&\mathcal{D}_{\mathcal{G}_2}\otimes I_{n_1}+I_{n_2}\otimes (\mathcal{D}_{\mathcal{G}_1}+I_{n_1})
        \end{bmatrix}$$
        Then the Signless Laplacian matrix is,\\
        $$\mathcal{Q}_{\mathcal{G}_1\boxtimes\mathcal{G}_2}=\begin{bmatrix}
            n_2(\mathcal{D}_{\mathcal{G}_1}+I_{n_1})+\mathcal{Q}_{\mathcal{G}_1}&\textbf{1}^T_{n_2}\otimes (I_{n_1}+\mathcal{A}_{\mathcal{G}_1})\\
            \textbf{1}_{n_2}\otimes (I_{n_1}+\mathcal{A}_{\mathcal{G}_1}))&\mathcal{Q}_{\mathcal{G}_2}\otimes I_{n_1}+I_{n_2}\otimes(\mathcal{D}_{\mathcal{G}_1}+I_{n_1})
        \end{bmatrix}$$
        The later part of the proof is the same as Theorem \ref{Th2}.
    \end{proof}

    \begin{corollary}
        Let $\mathcal{G}_1$ and $\mathcal{G}_2$ be two cospectral graphs and $\mathcal{G}$ be any arbitrary  graph, then 
            \begin{enumerate}
                \item $\mathcal{G}_1\boxtimes \mathcal{G}$ and $\mathcal{G}_2\boxtimes \mathcal{G}$ are signless Laplacian cospectral.
                \item $\mathcal{G}\boxtimes \mathcal{G}_1$ and $\mathcal{G}\boxtimes \mathcal{G}_2$ are signless Laplacian cospectral if $\chi_{Q_{\mathcal{G}_1}}(x)=\chi_{Q_{\mathcal{G}_2}}(x).$
            \end{enumerate}
    \end{corollary}
    \begin{proof}
        The proof is the same as Corollary \ref{C2}.
    \end{proof}

\section{Application}\label{S4}
    \subsection{Kirchhoff Index}
    \noindent In 1993 Klein and Randić~\cite{klein1993rd} proposed a novel concept known as resistance distance. This is based on the electric resistance of a network that corresponds to a graph, where the resistance distance between any two adjacent vertices is 1 ohm. This resistance distance, when summed over all pairs of vertices in a graph, serves as a new graph invariant. The Kirchhoff index, which is used to calculate electric resistance through Kirchhoff laws, is also defined for a graph $\mathcal{G}$ with $n(>2)$ vertices, as follows: $$Kf(\mathcal{G})=n\sum_{i=2}^n\frac{1}{\mu_i}.$$
    This renowned connection between the Laplaciathe n spectrum and Kirchhoff index was determined in 1996 by Zhu et al.~\cite{zhu1996extensions} and Gutman and Mohar~\cite{gutman1996quasi}.

    \begin{theorem}
        Let $\mathcal{G}_1$ be a $r_1$-regular graph of order $n_1$ and $\mathcal{G}_2$ be an arbitrary graph of order $n_2$. Also let the Laplacian spectrum be $0=\gamma_{i1}\leq\gamma_{i2}\leq\ldots\leq\gamma_{in_i};~ i=1,2.$ Then $$Kf(\mathcal{G}_1\boxtimes\mathcal{G}_2)=n_1(n_2+1)\Bigg[\sum_{i=2}^{n_2}\frac{n_1}{\gamma_{2i}+r_1+1}+\sum_{i=1}^{n_1}\frac{(r_1+1)(n_2+1)+\gamma_{1i}}{\gamma_{1i}(r_1+1)+\gamma_{1i}n_2(2+2r_1-\gamma_{1i})}\Bigg].$$
    \end{theorem}
    \begin{proof}
        From Theorem \ref{Th2} we have the Laplacian characteristics polynomial as\\
        $f(\mathcal{L}_{\mathcal{G}_1\boxtimes\mathcal{G}_2},x)=\prod_{i=1}^{n_2}(x-(r_1+1)-\gamma_{2i})^{n_1}\cdot\prod_{i=1}^{n_1}\Big[x-n_2(r_1+1)-\gamma_{1i}-\chi_{\mathcal{L}_{\mathcal{G}_2}}(x-(r_1+1))(r_1+1-\gamma_{1i})^2\Big].$\\
        Now, each row sum of the Laplacian matrix of $\mathcal{G}_2$ being $0$, the Laplacian coronal of $\mathcal{G}_2$ is given by $$\chi_{L_{\mathcal{G}_2}}=\frac{n_2}{x}.$$
        So, \begin{equation}\label{Eqn5}
            \begin{split}
                &f(\mathcal{L}_{\mathcal{G}_1\boxtimes\mathcal{G}_2},x)=\prod_{i=1}^{n_2}(x-(r_1+1)-\gamma_{2i})^{n_1}\cdot\prod_{i=1}^{n_1}\Big[x-n_2(r_1+1)-\gamma_{1i}-\frac{n_2}{x-r_1-1}\\&~~~~(r_1+1-\gamma_{1i})^2\Big]\\
                &=\prod_{i=1}^{n_2}(x-r_1-\gamma_{2i})^{n_1}\cdot\prod_{i=1}^{n_1}\frac{1}{x-r_1-1}\big[x^2-((r_1+1)(n_2+1)+\gamma_{1i})x+\gamma_{1i}(r_1+1)\\&~~~~+n_2\gamma_{1i}(2+2r_1-\gamma_{1i})\big].
            \end{split}
        \end{equation}
        Therefore, the roots of the above Laplacian characteristic polynomial are:
        \begin{enumerate}
            \item $r_1+1+\gamma_{22},r_1+1+\gamma_{23},\ldots,r_1+1+\gamma_{2n_2}$, each repeated $n_1$ times. [$(r_1+1+\gamma_{21})$ is not a root as $r_1+1$ is a pole of $\chi_{\mathcal{G}_2}(x-r_1-1)$].
            \item Roots of the polynomial $x^2-((r_1+1)(n_2+1)+\gamma_{1i})x+\gamma_{1i}(r_1+1)+n_2\gamma_{1i}(2+2r_1-\gamma_{1i}),$ for each $i=1,2,\ldots,n_1$.
        \end{enumerate}
        Let $\alpha$ and $\beta$ be the roots of the equation 
        \begin{equation}\label{Eqn6}
            x^2-((r_1+1)(n_2+1)+\gamma_{1i})x+\gamma_{1i}(r_1+1)+n_2\gamma_{1i}(2+2r_1-\gamma_{1i}).
        \end{equation}
        Then,
        \begin{equation*}
            \begin{split}
                \frac{1}{\alpha}+\frac{1}{\beta}&=\frac{\alpha+\beta}{\alpha\beta}\\
                &=\frac{(r_1+1)(n_2+1)+\gamma_{1i}}{\gamma_{1i}(r_1+1)+n_2\gamma_{1i}(2+2r_1-\gamma_{1i})}.
            \end{split}
        \end{equation*}
        Hence, the Kirchhoff index of $\mathcal{G}_1\boxtimes\mathcal{G}_2$ is 
        $$Kf(\mathcal{G}_1\boxtimes\mathcal{G}_2)=n_1(n_2+1)\Bigg[\sum_{i=2}^{n_2}\frac{n_1}{\gamma_{2i}+r_1+1}+\sum_{i=1}^{n_1}\frac{(r_1+1)(n_2+1)+\gamma_{1i}}{\gamma_{1i}(r_1+1)+\gamma_{1i}n_2(2+2r_1-\gamma_{1i})}\Bigg].$$ 
    \end{proof}

    \subsection{Spanning Tree}
        The concept of a spanning tree~\cite{cvetkovic1980spectra} refers to a tree that is also a subgraph of a given graph $\mathcal{G}$. The number of spanning trees for a graph $\mathcal{G}$ is represented by $t(\mathcal{G})$. If $\mathcal{G}$ is a connected graph with $n$ vertices and Laplacian eigenvalues $0=\gamma_1\leq\gamma_2\leq\ldots\leq\gamma_n$, then the number of spanning trees can be calculated using the formula: $$t(\mathcal{G})=\frac{\gamma_2(\mathcal{G})\gamma_3(\mathcal{G})\ldots\gamma_n(\mathcal{G})}{n}.$$

        \begin{theorem}
            Let $\mathcal{G}_1$ be a $r_1$-regular graph of order $n_1$ and $\mathcal{G}_2$ be an arbitrary graph of order $n_2$. Also let the Laplacian spectrum be $0=\gamma_{i1}\leq\gamma_{i2}\leq\ldots\leq\gamma_{in_i};~ i=1,2.$ Then the number of spanning trees is given by, $$t(\mathcal{G}_1\boxtimes\mathcal{G}_2)=\frac{1}{n_1(n_2+1)}\Big[\prod_{i=2}^{n_2}(r_1+1+\gamma_{2i})^{n_1}\cdot\prod_{i=1}^{n_1}\gamma_{1i}(r_1+1)+n_2\gamma_{1i}(2+2r_1-\gamma_{1i})\Big].$$
        \end{theorem}
        \begin{proof}
            Let $\alpha,$ $\beta$ be two roots of equation (\ref{Eqn6}) i.e., $x^2-((r_1+1)(n_2+1)+\gamma_{1i})x+\gamma_{1i}(r_1+1)+n_2\gamma_{1i}(2+2r_1-\gamma_{1i}).$
            Then, $\alpha\beta=\gamma_{1i}(r_1+1)+n_2\gamma_{1i}(2+2r_1-\gamma_{1i}).$\\
            Applying the definition of the number of spanning trees in equation (\ref{Eqn5}), we obtain: $$t(\mathcal{G}_1\boxtimes\mathcal{G}_2)=\frac{1}{n_1(n_2+1)}\Big[\prod_{i=2}^{n_2}(r_1+1+\gamma_{2i})^{n_1}\cdot\prod_{i=1}^{n_1}\gamma_{1i}(r_1+1)+n_2\gamma_{1i}(2+2r_1-\gamma_{1i})\Big].$$
        \end{proof}

\subsection{Equienergetic Graph}
    \noindent In this section, we describe a method for constructing a non-cospectral equienergetic graph using the product graph.\\
    We know for a graph $\mathcal{G}$ the coronal of $\mathcal{G}$ is 
    \begin{equation*}
        \begin{split}
            \chi_{A_{\mathcal{G}}}&=\textbf{1}_n^T(xI_n-\mathcal{A}_{\mathcal{G}})^{-1}\textbf{1}_n\\
            &=\frac{\textbf{1}_n^TAdj(xI_n-\mathcal{A}_{\mathcal{G}})\textbf{1}_n}{\det(xI_n-\mathcal{A}_{\mathcal{G}})}\\
            &=\frac{p(A_{\mathcal{G}},x)}{f(A_{\mathcal{G}},x)}
        \end{split}
    \end{equation*}
    where $p(A_{\mathcal{G}},x)$ is a polynomial of degree $n-1$ and $f(A_{\mathcal{G}},x)$ is the characteristics polynomial of $A_{\mathcal{G}}$. Now, if the greatest common divisor of the above fraction is not a constan,t then we can rewrite it as,
    $$\chi_{A_{\mathcal{G}}}=\frac{P_{d-1}(A_{\mathcal{G}},x)}{F_d(A_{\mathcal{G}},x)}$$
    where $P_d(A_{\mathcal{G}},x)$ and $F_d(A_{\mathcal{G}},x)$ are polynomials of degree $d-1$ and $d$ respectively and $$gcd\Big(p(A_{\mathcal{G}},x),f(A_{\mathcal{G}},x)\Big)=R_{n-d}(x)$$.\\
    From Theorem \ref{Th1} we have 
    \begin{equation*}
        \begin{split}
            f(\mathcal{A}_{\mathcal{G}_1\boxtimes\mathcal{G}_2},x)&=\prod_{i=1}^{n_2}(x-\lambda_{2i})^{n_1}\cdot\prod_{i=1}^{n_1}\Big[x-\lambda_{1i}-\chi_{A_{\mathcal{G}_2}}(x)(1+\lambda_{1i})^2 \Big]\\
            &=f(A_{\mathcal{G}_2},x)^{n_1}\cdot \prod_{i=1}^{n_1}\Big[x-\lambda_{1i}-\frac{P_{d_2-1}(A_{\mathcal{G}_2},x)}{F_{d_2}(A_{\mathcal{G}_2},x)}(1+\lambda_{1i})^2\Big]\\
            &=(R_{n_2-d_2}(x))^{n_1}\cdot \prod_{i=1}^{n_1}\Big[F_{d_2}x-F_{d_2}\lambda_{1i}-P_{d_2-1}(1+\lambda_{1i})^2\Big]
        \end{split}
    \end{equation*}

    \begin{theorem}\label{Th5}
        Let $\mathcal{G}_1$ and $\mathcal{G}_2$ be two non-cospectral equienergetic graph of order $n$ with $\chi_{\mathcal{A}_{\mathcal{G}_1}}(x)=\chi_{A_{\mathcal{G}_2}}(x)$. Then for an arbitrary graph $\mathcal{G}$ of order $m$, the graphs $\mathcal{G}\boxtimes\mathcal{G}_1$ and $\mathcal{G}\boxtimes\mathcal{G}_2$ are non-cospectral equienergetic.
    \end{theorem}
    \begin{proof}
        The coronals of $\mathcal{G}_1$ and $\mathcal{G}_2$ are the same, so following the equation given above, we can write the characteristics polynomial equations as
        $$f(\mathcal{A}_{\mathcal{G}_1},x)=R_{n-d}(x)F_d(x)$$
        and 
        $$f(A_{\mathcal{G}_2},x)=R'_{n-d}(x)F_d(x)$$
        Clearly, $R_{n-d}(x)\neq R'_{n-d}(x)$ as $\mathcal{G}_1$ and $\mathcal{G}_2$ are non-cospectral.\\
        The characteristics polynomial of $\mathcal{G}\boxtimes\mathcal{G}_1$ and $\mathcal{G}\boxtimes\mathcal{G}_2$ are
        \begin{equation}\label{Eqn7}
            f(A_{\mathcal{G}\boxtimes\mathcal{G}_1},x)=(R_{n-d}(x))^{m}\cdot \prod_{i=1}^{m}\Big[F_{d}(x)x-F_{d}(x)\lambda_{i}-P_{d-1}(x)(1+\lambda_{i})^2\Big]
        \end{equation}
        \begin{equation}\label{Eqn8}
            f(A_{\mathcal{G}\boxtimes\mathcal{G}_2},x)=(R'_{n-d}(x))^{m}\cdot \prod_{i=1}^{m}\Big[F_{d}(x)x-F_{d}(x)\lambda_{i}-P_{d-1}(x)(1+\lambda_{i})^2\Big],
        \end{equation}
        where $\lambda_i's$ are the eigenvalues of $A_{\mathcal{G}}$.\\
        Let the roots of the polynomial $F_d(x)$ be $f_1,f_2,\ldots,f_d$ and let the roots of $R_{n-d}(x)$ and $R'_{n-d}(x)$ be $r_1,r_2,\ldots,r_{n-d}$ and $r'_1,r'_2,\ldots,r'_{n-d}$ respectively. Given the graphs $\mathcal{G}_1$ and $\mathcal{G}_2$ are non-cospectral and equienergetic, we have,\\
        $$\sum_{i=1}^d|f_i|+\sum_{i=1}^{n-d}|r_i|=\mathcal{E}_{\mathcal{G}_1}=\mathcal{E}_{\mathcal{G}_2}=\sum_{i=1}^d|f_i|+\sum_{i=1}^{n-d}|r'_i|.$$
        So,
        $$\sum_{i=1}^{n-d}|r_i|=\sum_{i=1}^{n-d}|r'_i|.$$
        The product part of equation (\ref{Eqn7}) and (\ref{Eqn8}) are same and is a $m(d+2)$ degree polynomial. Let the root of this polynomial be $s_1,s_2,\ldots,s_{m(d+2)}$. Then the energy of $A_{\mathcal{G}\boxtimes\mathcal{G}_1}$ and $A_{\mathcal{G}\boxtimes\mathcal{G}_2}$ are,
        $$\mathcal{E}_{A_{\mathcal{G}\boxtimes\mathcal{G}_1}}=m\sum_{i=1}^{n-d}|r_i|+\sum_{i=1}^{m(d+2)}|s_i|$$ and 
        $$\mathcal{E}_{A_{\mathcal{G}\boxtimes\mathcal{G}_1}}=m\sum_{i=1}^{n-d}|r'_i|+\sum_{i=1}^{m(d+2)}|s_i|.$$
        Therefore,
        $$\mathcal{E}_{A_{\mathcal{G}\boxtimes\mathcal{G}_1}}=m\sum_{i=1}^{n-d}|r_i|+\sum_{i=1}^{m(d+2)}|s_i|=m\sum_{i=1}^{n-d}|r'_i|+\sum_{i=1}^{m(d+2)}|s_i|=\mathcal{E}_{A_{\mathcal{G}\boxtimes\mathcal{G}_1}}.$$
        Clearly, from the above equation $\mathcal{G}\boxtimes\mathcal{G}_1$ and $\mathcal{G}\boxtimes\mathcal{G}_2$ are equienergetic non-cospectral graphs.
    \end{proof}

    \begin{corollary}
        Let $\mathcal{G}_1$ and $\mathcal{G}_2$ be two non-cospectral equienergetic $r$-regular graphs. Then for an arbitrary graph $\mathcal{G}$, the product graphs $\mathcal{G}\boxtimes\mathcal{G}_1$ and $\mathcal{G}\boxtimes\mathcal{G}_2$ are equienergetic non-cospectral graphs.
    \end{corollary}
    \begin{proof}
        Since both the graphs $\mathcal{G}_1$ and $\mathcal{G}_2$ are regular, they have the same coronal \cite{mcleman2011spectra}, also $\mathcal{G}_1$ and $\mathcal{G}_2$ are non-cospectral equienergetic. Hence $\mathcal{G}\boxtimes\mathcal{G}_1$ and $\mathcal{G}\boxtimes\mathcal{G}_2$ are equienergetic non-cospectral graphs.
    \end{proof}

\subsection{Integral Graph}
    \noindent In this section, we discuss the condition for the graph product to be Integral.

    \begin{theorem}
        Let $\mathcal{G}_1$ and $\mathcal{G}_2$ be two graphs of orders $n_1$ and $n_2$, respectively. Then $\mathcal{G}_1\boxtimes \mathcal{G}_2$ is integral graph if and only if $\mathcal{G}_2$ is Integral and for $i=1,2,\ldots,n_1$, the roots of the polynomial $\big[x-\lambda_{1i}-\chi_{A_{\mathcal{G}_2}}(x)(1+\lambda_{1i})^2\big]$ are integers.
    \end{theorem}
    \begin{proof}
        The proof is obvious from Theorem~\ref{Th1}.
    \end{proof}

    \section*{Acknowledgement}
        We would like to acknowledge the National Institute of Technology Sikkim for awarding Bishal Sonar a doctoral fellowship.
    \section*{Disclosure statement}
        The authors report there are no competing interests to declare.
    \bibliographystyle{abbrv}
    \bibliography{main.bib}    
\end{document}